\documentclass{amsart}
\usepackage{amssymb, amsmath, latexsym, mathabx, dsfont}

\renewcommand{\baselinestretch}{\baselinestretch}
\renewcommand{\baselinestretch}{1.1}
\numberwithin{equation}{section}

\newtheorem{thm}{Theorem}[section]

\newcommand{\Mod}[1]{\ (\mathrm{mod}\ #1)}

\begin{document}

\title{Universal sums of generalized pentagonal numbers}
\author{Jangwon Ju}

\address{Department of Mathematics, University of Ulsan, Ulsan 44610, Republic of Korea}
\email{jjw@snu.ac.kr}
\thanks{}



\begin{abstract} For an integer $x$, an integer of the form $P_5(x)=\frac{3x^2-x}2$ is called a generalized pentagonal number.
 For positive integers $\alpha_1,\dots,\alpha_k$, a sum $\Phi_{\alpha_1,\dots,\alpha_k}(x_1,x_2,\dots,x_k)=\alpha_1P_5(x_1)+\alpha_2P_5(x_2)+\cdots+\alpha_kP_5(x_k)$ of generalized pentagonal numbers is called {\it universal} if  $\Phi_{\alpha_1,\dots,\alpha_k}(x_1,x_2,\dots,x_k)=N$ has an integer solution $(x_1,x_2,\dots,x_k) \in \mathbb Z^k$ for any non-negative integer $N$.   In this article, we prove that there are exactly $234$ proper universal sums of generalized  pentagonal numbers. 
Furthermore, the ``pentagonal theorem of $109$" is proven, which states that an arbitrary sum $\Phi_{\alpha_1,\dots,\alpha_k}(x_1,x_2,\dots,x_k)$  is universal if and only if it represents the integers $1, 3, 8, 9, 11, 18, 19,  25, 27, 43, 98$, and $109$.
\end{abstract}

\maketitle

\section{Introduction}
Let $m$ be any positive integer greater than equal to three.  {\it A polygonal number of order $m$} (or an {\it $m$-gonal number}) is defined by the integer
$$
P_m(x)=\frac{(m-2)x^2-(m-4)x}{2}
$$
for some non-negative integer $x$.  If $x$ is an arbitrary integer, then we say $P_m(x)$ is  {\it a generalized polygonal number of order $m$} (or  {\it a generalized $m$-gonal number}).

A famous assertion of Fermat says that every non-negative integer is written as a sum of $m$ polygonal numbers of order $m$. 
In 1770, Lagrange proved that Fermat's assertion holds for $m=4$. In 1796, Gauss proved that the assertion holds for $m=3$. Finally, in 1813, Cauchy proved the assertion completely.
The Fermat polygonal number theorem stated above was generalized in many directions.

As a natural generalization of Lagrange's four square theorem, Ramanujan provided a list of 55 candidates of quaternary diagonal integral quadratic forms that represent all non-negative integers. 
In \cite{dick}, Dickson pointed out the diagonal quaternary quadratic form $x^2+2y^2+5z^2+5z^2$ in Ramanujan's list does not represent $15$ and confirmed that all the other 54 diagonal quadratic forms in the list represent all non-negative integers.
Conway, Miller and Schneeberger proved the so-called ``15-theorem", which states that a positive definite integral quadratic form represents  all non-negative integers if and only if it represents the integers $1,2,3,5,6,7,10,14,~\text{and}~15$, irrespective of its rank (for details, see \cite{b}).
 Recently,  Bhargava and Hanke \cite{BH} generalized the ``$15$-theorem" to the arbitrary positive definite integer-valued quadratic forms by proving the so-called ``290-theorem", which states that    
a positive-definite integer-valued quadratic form represents all non-negative integers if and only if it represents the integers
$$
\begin{array} {ll}
1,\ 2, \ 3,\ 5,\ 6,\ 7,\ 10,\ 13,\ 14,\ 15,\ 17,\ 19,\ 21,\ 22,\ 23,\ 26,\ 29,\\
30,\ 31,\ 34,\ 35,\ 37,\ 42,\ 58,\ 93,\ 110,\ 145,\ 203, \ \  \text{and} \  \ 290.
\end{array}
$$
Here, a quadratic form 
$$
f(x_1,x_2,\dots,x_n)=\sum_{1 \le i, j\le n} a_{ij} x_ix_j \quad (a_{ij}=a_{ji})
$$
 is called {\it integral} if $a_{ij}\in \mathbb Z$ for any $i,j$, and is called {\it integer-valued} if $a_{ii}\in \mathbb Z$ and $a_{ij}+a_{ji}\in\mathbb Z$ for any $i,j$.

For positive integers $\alpha_1$, $\alpha_2$, $\dots$, $\alpha_k$, a sum $\alpha_1P_m(x_1)+\alpha_2P_m(x_2)+\cdots+\alpha_kP_m(x_k)$ of generalized polygonal numbers of order $m$ is called {\it universal} if the diophantine equation 
$$
\alpha_1P_m(x_1)+\alpha_2P_m(x_2)+\cdots+\alpha_kP_m(x_k)=N
$$
 has an integer solution $(x_1,x_2,\dots,x_k)\in\mathbb{Z}^k$ for any non-negative integer $N$.
 
In 1862, Liouville generalized Gauss' triangular theorem by proving that a ternary sum $\alpha_1P_3(x_1)+\alpha_2P_3(x_2)+\alpha_3P_3(x_3)$ of triangular numbers is universal if and only if $(\alpha_1,\alpha_2,\alpha_3)$ is one of the following triples:
$$
(1,1,1),\quad(1,1,2),\quad(1,1,4),\quad(1,1,5),\quad(1,2,2),\quad(1,2,3),\quad(1,2,4).
$$
Recently, Bosma and Kane  \cite{bk} proved, so-called, the ``triangular theorem of eight" which states that an arbitrary sum $\alpha_1P_3(x_1)+\alpha_2P_3(x_2)+\cdots+\alpha_kP_3(x_k)$ of triangular numbers is universal if and only if it represents the integers $1,2,4,5$, and $8$.
One may consider this as a natural  generalization of the ``$15$-theorem". 
 
In \cite{octagonal} and \cite{sun}, it was proven that there are exactly $40$ quaternary universal sums of generalized octagonal numbers.
Furthermore, the ``octagonal theorem of sixty" was proven  in \cite{octagonal}, which states that  an arbitrary sum $a_1 P_8(x_1)+a_2P_8(x_2)+\cdots+a_kP_8(x_k)$ of  generalized octagonal numbers is universal if and only if it represents 
$$
1,\ 2, \ 3,\  4,\ 6,\  7,\ 9,\ 12,\ 13,\ 14,\ 18, \ \text{and} \  60.
$$ 

In \cite{pentagonal}, Oh proved that there are exactly $20$ ternary universal sums  of generalized pentagonal numbers, which was conjectured by Sun in \cite{sun2}. 
In fact, a ternary sum $\alpha_1P_5(x_1)+\alpha_2P_5(x_2)+\alpha_3P_5(x_3)$ of generalized pentagonal numbers is universal if and only if  $(\alpha_1,\alpha_2,\alpha_3)$ is one of the triples:
\begin{equation}\label{pent}
\begin{array}{ll}
(\alpha_1,\alpha_2,\alpha_3)= &(1,1,s),\quad \text{for $1\le s \le 10$ and $s \ne 7$,} \\
             ~                &(1,2,2), \ (1,2,3), \ (1,2,4),  \ (1,2,6), \ (1,2,8),  \\               
               ~              &(1,3,3), \ (1,3,4), \  (1,3,6), \ (1,3,7), \ (1,3,8), \ (1,3,9).
\end{array}
\end{equation}

In this paper, we prove that there are exactly $234$ proper universal sums of generalized  pentagonal numbers. 
Furthermore, we prove the ``pentagonal theorem of $109$" which states that an arbitrary sum $\alpha_1P_5(x_1)+\alpha_2P_5(x_2)+\cdots+\alpha_kP_5(x_k)$ of generalized pentagonal numbers is universal if and only if it represents the integers
$$
1,\ 3,\ 8,\ 9,\ 11,\ 18,\ 19,\ 25,\ 27,\ 43,\ 98,\ \text{and}\ 109.
$$
This  might also be considered as a natural generalization of the ``$15$-theorem".

Assume that a quadratic form $f(x_1,x_2,\dots,x_n)=\sum_{1 \le i, j\le n} a_{ij} x_ix_j \ (a_{ij}=a_{ji})$ 
is  positive definite and integral.
 The symmetric matrix corresponding to $f$ is defined by $M_f=(a_{ij})$. 
Note that our definition on the Gram matrix is slightly different from $\left(\frac{\partial^2f}{\partial x_i\partial x_j}\right)$.
If $a_{ij}=0$ for any $i\ne j$, then we simply write 
$$
f=\langle a_{11},a_{22},\dots,a_{nn}\rangle.
$$ 
For a non-negative integer $N$, if the diophantine equation $f(x_1,x_2,\dots,x_n)=N$ has an integer solution, then we say $N$ is represented by $f$. 
The genus of $f$, denoted by $\text{gen}(f)$, is the set of all quadratic forms that are isometric to $f$ over $\mathbb{Z}_p$ for any prime $p$. 
The number of isometry classes in $\text{gen}(f)$ is called the class number of $f$, and denoted by $h(f)$.

Any unexplained notations and terminologies can be found in \cite{ki} or \cite{om}.

\section{General tools}
Let $\alpha_1,\alpha_2\dots,\alpha_k$ be positive integers. Recall that a sum 
$$
\Phi_{\alpha_1,\alpha_2,\cdots,\alpha_k}(x_1,x_2,\dots,x_k)=\alpha_1P_5(x_1)+\alpha_2P_5(x_2)+\cdots+\alpha_kP_5(x_k)
$$
of generalized pentagonal numbers is said to be universal if the diophantine equation
$$
\Phi_{\alpha_1,\alpha_2,\cdots,\alpha_k}(x_1,x_2,\dots,x_k)=N
$$
has an integer solution $(x_1,x_2,\dots,x_k)\in\mathbb{Z}^k$ for any non-negative integer $N$.
We say the sum $\Phi_{\alpha_1,\alpha_2,\dots,\alpha_k}$ of generalized pentagonal numbers is {\it proper universal} if $\Phi_{\alpha_1,\alpha_2,\dots,\alpha_k}$ is universal  
and there does not exist a proper subset $\{i_1,i_2,\dots,i_u\} \subset \{1,2,\dots,k\}$
such that the partial sum $\Phi_{\alpha_{i_1},\dots,\alpha_{i_u}}$  is universal. 

One may easily show that the equation $\Phi_{\alpha_1,\alpha_2,\dots,\alpha_k}(x_1,x_2,\dots,x_k)=N$ has an integer solution if and only if the equation
$$
\alpha_1(6x_1-1)^2+\alpha_2(6x_2-1)^2+\cdots+\alpha_k(6x_k-1)^2=
24N+\alpha_1+\alpha_2+\cdots+\alpha_k
$$
has an integer solution. Note that if an integer $v$ is relatively prime to $6$, then one of the integers $v$ or $-v$ is congruent to $-1$ modulo $6$. 
Therefore, the sum $\Phi_{\alpha_1,\alpha_2,\dots,\alpha_k}$ of generalized pentagonal numbers is universal if and only if 
the diophantine equation
$$
\alpha_1x_1^2+\alpha_2x_2^2+\dots+\alpha_kx_k^2=24N+\alpha_1+\alpha_2+\cdots+\alpha_k
$$
has an integer solution $(x_1,x_2,\dots,x_k)\in\mathbb{Z}^k$ such that $\text{gcd}(x_1x_2\cdots x_k,6)=1$ for any non-negative integer $N$. 

In some particular cases, representations of quadratic forms with some congruence condition correspond to representations of a subform which is suitably taken (for details, see \cite{poly}).

In \cite{poly}, \cite{regular}, and \cite{pentagonal}, we developed a method on determining whether or not integers in an arithmetic progression are represented by some particular ternary quadratic form. 
We briefly introduce this method for those who are unfamiliar with it. 
 
 Let $d$ be a positive integer and let $a$ be a non-negative integer $(a\leq d)$. 
We define 
$$
S_{d,a}=\{dn+a \mid n \in \mathbb N \cup \{0\}\}.
$$
For integral ternary quadratic forms $f,g$, we define
$$
R(g,d,a)=\{v \in (\mathbb{Z}/d\mathbb{Z})^3 \mid vM_gv^t\equiv a \ (\text{mod }d) \}
$$
and
$$
R(f,g,d)=\{T\in M_3(\mathbb{Z}) \mid  T^tM_fT=d^2M_g \}.
$$
A coset (or, a vector in the coset) $v \in R(g,d,a)$ is said to be {\it good} with respect to $f,g,d \text{ and }a$ if there is a $T\in R(f,g,d)$ such that $\frac1d \cdot vT^t \in \mathbb{Z}^3$.  
The set of all good vectors in $R(g,d,a)$ is denoted by $R_f(g,d,a)$.    
If  $R(g,d,a)= R_f(g,d,a)$, we write  $g\prec_{d,a} f$. 
If $g\prec_{d,a} f$, then by Lemma 2.2 of \cite{regular},  we have 
\begin{equation}\label{gt}
S_{d,a}\cap Q(g) \subset Q(f).
\end{equation}
Note that $R(g,d,a)$ and $R(f,g,d)$ are both finite sets, so the determination whether $g\prec_{d,a}f$ or not could be computed in a finite time.
In general, if $d$ is large, then it is not easy to determine by hand whether or not two finite sets $R(g,d,a)$ and  $R_f(g,d,a)$ are equal.  
 A computer program for this based on MAPLE is available upon request to the author.

\section{Universal sums of generalized pentagonal numbers}

In this section, we determine all proper universal sums of generalized pentagonal numbers. 
Furthermore, we give an effective criterion on the universality of an arbitrary sum of generalized pentagonal numbers, which is a natural generalization of the ``$15$-theorem".

For positive integers $\alpha_1,\alpha_2,\dots,\alpha_k$, let 
$$
\Phi_{\alpha_1,\alpha_2,\dots,\alpha_k}(x_1,x_2,\dots,x_k)=\alpha_1P_5(x_1)+\alpha_2P_5(x_2)+\cdots+\alpha_kP_5(x_k)
$$ 
be a sum of generalized pentagonal numbers. 
For an integer $N$, if the diophantine equation 
$$
\Phi_{\alpha_1,\alpha_2,\dots,\alpha_k}(x_1,x_2,\dots,x_k)=N
$$ 
has an integer  solution $(x_1,x_2,\dots,x_k)\in\mathbb{Z}^k$, then we say the sum $\Phi_{\alpha_1,\alpha_2,\dots,\alpha_k}$ represents $N$, and we write $N\rightarrow\Phi_{\alpha_1,\alpha_2,\dots,\alpha_k}$. When the sum $\Phi_{\alpha_1,\alpha_2,\dots,\alpha_k}$ of generalized pentagonal numbers is not universal, the least positive integer that is not represented by $\Phi_{\alpha_1,\alpha_2,\dots,\alpha_k}$  is called the {\it truant} of $\Phi_{\alpha_1,\alpha_2,\dots,\alpha_k}$. 

 Since all of ternary universal sums of generalized pentagonal numbers are completely determined as in \eqref{pent}, we first consider the quaternary case.
  
\begin{thm}\label{quaternary}
There are exactly $90$ quaternary proper universal sums of generalized pentagonal numbers.
\end{thm}

\begin{proof}
For positive integers $\alpha_1,\alpha_2,\alpha_3,\alpha_4$, let $\Phi_{\alpha_1,\alpha_2,\alpha_3,\alpha_4}$ be a quaternary proper universal sum of generalized pentagonal numbers. 
Without loss of generality, we may assume that $\alpha_1\leq\alpha_2\leq \alpha_3 \leq\alpha_4$.
Since $1\rightarrow \Phi_{\alpha_1,\alpha_2,\alpha_3,\alpha_4}$, we have $\alpha_1=1$. 
Since the truant of $\Phi_1$ is 3 and $3\rightarrow\Phi_{1,\alpha_2,\alpha_3,\alpha_4}$, we have $\alpha_2=1,2$, or $3$. If $(\alpha_1,\alpha_2)=(1,1)$, then $1\leq\alpha_3\leq11$, for the truant of $\Phi_{1,1}$ is $11$ and $11\rightarrow\Phi_{1,1,\alpha_3,\alpha_4}$. Similarly, for $i=2$ or $3$,
if $(\alpha_1,\alpha_2)=(1,i)$, then $i\leq\alpha_3\leq i+6$,   for the truant of $\Phi_{1,i}$ is $i+6$ and $i+6\rightarrow\Phi_{1,i,\alpha_3,\alpha_4}$.
If $(\alpha_1,\alpha_2,\alpha_3)\not\eq(1,1,7),(1,1,11),(1,2,5),(1,2,7)$, and $(1,3,5)$, then each ternary sum $\Phi_{\alpha_1,\alpha_2,\alpha_3}$ of generalized pentagonal numbers is universal by \eqref{pent}.
Note that the truant $t$ of each $\Phi_{\alpha_1,\alpha_2,\alpha_3}$ is 
$$
t=
\begin{cases}
25& \quad\text{if~} (\alpha_1,\alpha_2,\alpha_3)=(1,1,7),\\
43& \quad\text{if~} (\alpha_1,\alpha_2,\alpha_3)=(1,1,11),\\
18& \quad\text{if~} (\alpha_1,\alpha_2,\alpha_3)=(1,2,5),\\
27& \quad\text{if~} (\alpha_1,\alpha_2,\alpha_3)=(1,2,7),\\
19& \quad\text{if~} (\alpha_1,\alpha_2,\alpha_3)=(1,3,5).
\end{cases}
$$
Therefore, $\alpha_3\leq\alpha_4\leq t$ for each possible case, where $t$ is the integer given above.
We show that there are exactly $90$ quaternary universal sums of generalized pentagonal numbers among the above candidates. For the complete list of proper universal sums of generalized pentagonal numbers, see Table 3.1.
%%%%%%%%%%%%%%%%%%%%%%%%%%%%%%%%%%%%%%%%%%%%%%%%%%
%%%%%%%%%%%%%%%%%%%%%%%%%%%%%%%%%%%%%%%%%%%%%%%%%%%%%%%%%
\vskip 1pc
\begin{center}
\renewcommand{\arraystretch}{1}
\begin{tabular}{ccrcl}\multicolumn{5}{c}{{Table 3.1. Proper universal sums of generalized pentagonal numbers}}\\
\hline
& Sums&&$\alpha$ &\\ \hline
(3-1) & $\Phi_{1,1,\alpha}$ & \hspace{15mm}$1\leq\alpha\leq10$ &{and}  &$\alpha\not\eq7$\\
(3-2) & $\Phi_{1,2,\alpha}$ & $2\leq\alpha\leq8$ &{and}   &$\alpha\not\eq5,7$\\
(3-3) & $\Phi_{1,3,\alpha}$ & $3\leq\alpha\leq9$ &{and}   &$\alpha\not\eq5$\\
\hline
(4-1) & $\Phi_{1,1,7,\alpha}$ & $7\leq\alpha\leq25$ &{and}   &$\alpha\not\eq8,9,10$\\
(4-2) & $\Phi_{1,1,11,\alpha}$ & $11\leq\alpha\leq43$ &{and}   &$\alpha\not\eq22,33$\\
(4-3) & $\Phi_{1,2,5,\alpha}$ & $5\leq\alpha\leq18$ &{and}   &$\alpha\not\eq6,8$\\
(4-4) & $\Phi_{1,2,7,\alpha}$ & $7\leq\alpha\leq27$ &{and}   &$\alpha\not\eq8$\\
(4-5) & $\Phi_{1,3,5,\alpha}$ & $5\leq\alpha\leq19$ &{and}   &$\alpha\not\eq6,7,8,9$\\
\hline
(5-1) & $\Phi_{1,1,11,22,\alpha}$ & $\alpha=22,33$ &{or}   &$44\leq\alpha\leq98$\\
(5-2)& $\Phi_{1,1,11,33,\alpha}$ & $\alpha=33$ &{or}   &$44\leq\alpha\leq109$\\
\hline
\end{tabular}
\end{center}

In the case (4-1), we will explain how our method works in detail. Since everything is quite similar to this for all the other cases, we briefly provide all parameters needed for computations, in the remaining cases.
\bigskip 

\noindent\textbf{Case (4-1) $(\alpha_1,\alpha_2,\alpha_3)=(1,1,7)$}. It is enough to show that for any $\alpha$ such that $7\le \alpha \le 25$ and $\alpha \ne 8,9,10$, the equation
$$
x^2+y^2+7z^2+\alpha t^2=24N+9+\alpha
$$ 
has an integer solution $(x,y,z,t)\in\mathbb{Z}^4$ such that $\text{gcd}(xyzt,6)=1$ for any non-negative integer $N$.

Since the proofs are quite similar to each other, we only provide the proof of the case when $\alpha=21$. Let $24N+30=7^{2s}(24n+6)$ for some non-negative integers $n,s$ such that $24n+6\not\equiv0\Mod{7^2}$. 
For the case when $n=0$, note that for any $s\ge1$,
$$
(7^s)^2+(7^s)^2+7(5\cdot7^{s-1})^2+21(7^{s-1})^2=6\cdot7^{2s}
$$ 
If $1\leq n\leq 105$, then one may directly show that the equation
$$
x^2+y^2+7z^2+21t^2=24n+6
$$ 
has an integer solution $(x,y,z,t)\in\mathbb{Z}^4$ such that $\text{gcd}(xyzt,6)=1$.
Therefore, we assume that $n\geq106$.
Note that the genus of $f(x,y,z)=x^2+y^2+7z^2$ consists of 
$$
M_f=\langle1,1,7\rangle\quad\text{and}\quad M_2=\langle1\rangle\perp\begin{pmatrix}2&1\\1&4\end{pmatrix}.
$$
One may easily show that every non-negative integer not of the form $k\cdot 7^{2l+1}$ for any non-negative integers $l,k$ such that $\left(\frac k7\right)=-1$ is represented by $M_f$ or $M_2$ by 102:5 of \cite{om}, for  it is represented by $M_f$ over $\mathbb{Z}_p$ for any prime $p$.

Let $g$ be the quadratic form associated to $M_2$. 
Note that $R(g,4,1)$ consists of $24$ cosets. Actually, it is the union of following sets:
$$
\begin{array}{l}
R_1=\{(v_1,v_2,v_3)\in (\mathbb{Z}/4\mathbb{Z})^3\mid v_1\equiv1\Mod2,~v_2\equiv0\Mod2\},\\
R_2=\{(v_1,v_2,v_3)\in (\mathbb{Z}/4\mathbb{Z})^3\mid v_1\equiv v_2\equiv v_3\equiv 1 \Mod2\}.
\end{array}
$$
Furthermore, note that $R(f,g,4)$ consists of $64$ matrices and in particular, it contains the following two matrices:
$$
T_1=\begin{pmatrix}0&2&8\\4&0&0\\0&-2&0\end{pmatrix}\quad\text{and}\quad
T_2=\begin{pmatrix}0&2&-6\\4&0&0\\0&-2&-2\end{pmatrix}.
$$
For each $i=1,2$, one may easily show that for any $(v_1,v_2,v_3)\in R_i$,
$$
\frac14\cdot(v_1,v_2,v_3)\cdot T_i^t\in\mathbb{Z}^3.
$$
Therefore, $g\prec_{4,1}f$ holds; hereafter, we simply abbreviate $M_2\prec_{4,1}M_f$ when the associated quadratic forms satisfy the relation $g\prec_{4,1}f$.
Then by \eqref{gt}, every positive integer congruent to $1$ modulo $4$ which is represented by $g$ is also represented by $f$.
As a sample, for $(1,5,1)\in\mathbb{Z}^3$, note that $g(1,5,1)=65$ and $(1,5,1)\equiv(1,1,1)\Mod4\in R_f(g,4,1)$.
Then 
$$
f(1,1,-3)=f\left(\frac14\cdot(1,5,1)\cdot T_2^t\right)=g(1,5,1)=65.
$$

One may easily show that there is an integer $d\in\{1,5,7,11\}$ such that $24n+6-21d^2\equiv1\Mod4$ and  $24n+6-21d^2\not\eq k\cdot 7^{2l+1}$. Furthermore, since we are assuming that $n\ge 106$,  $24n+6-21d^2$ is a positive integer.
Therefore, the equation 
$$
x^2+y^2+7z^2=24n+6-21d^2
$$ 
has an integer solution $(x,y,z)=(a,b,c)\in\mathbb{Z}^3$ by \eqref{gt}.
Assume that $a\equiv b\equiv c\equiv 0\Mod3$. We may assume that $b\equiv c\Mod4$. 
If we define
$$
\tau=\frac16\begin{pmatrix}2&5&7\\-2&-2&14\\-2&1&-1\end{pmatrix},
$$ 
then one may easily check that $\tau(a,b,c)^t=(a_1,b_1,c_1)^t$ is also an integer solution of the equation
\begin{equation}\label{117}
x^2+y^2+7z^2=24n+6-21d^2
\end{equation} 
such that $b_1\equiv c_1\Mod4$.
Therefore, there is a positive integer $m$ such that $\tau^m(a,b,c)^t=(a_m,b_m,c_m)^t$ is an integer solution of Equation \eqref{117} and one of whose component is not divisible by $3$ or 
for any positive integer $m$, $\tau^m(a,b,c)^t=(a_m,b_m,c_m)^t$  is an integer solution of Equation \eqref{117}  each of whose component is divisible by $3$.
Since there are only finitely many integer solution of Equation \eqref{117} 
and $\tau$ has an infinite order, the latter is impossible unless $(a,b,c)$ is an eigenvector of $\tau$. Note that the $\pm(1,-3,1)$ are the only integral primitive eigenvectors of $\tau$. 
Since $u^2+(-3u)^2+7u^2=(-u)^2+(-3u)^2+7u^2$ for any $u\in\mathbb{Z}$, we may assume that $(a,b,c)$ is not an eigenvector of $\tau$. Therefore, the equation 
$$
x^2+y^2+7z^2=24n+6-21d^2
$$ 
has an integer solution $(x,y,z)=(a,b,c)\in\mathbb{Z}^3$ such that $abc\not\equiv0\Mod 3$. 
We may assume that $a\equiv1\Mod2$. 
If $b\equiv c\equiv 1\Mod2$, then we are done. Assume $b\equiv c\equiv 0\Mod2$. Note that $b\equiv c\Mod4$. Let $b=2i$ and $c=2j$. Assume $i\equiv j\equiv1\Mod2$. 
Then 
$$
b^2+7c^2=\left(\frac{3i\pm7j}{2}\right)^2+7\left(\frac{i\mp3j}{2}\right)^2
$$
and one of the integer $\frac{i-3j}{2}$ or $\frac{i+3j}{2}$ is relatively prime to $6$. Therefore, there are integers $b_1, c_1$ such that $b^2+7c^2=b_1^2+7c_1^2$ and $\text{gcd}(b_1c_1,6)=1$. Assume $i\equiv j\equiv 0\Mod2$. There are integers $i_1,j_1$ such that 
$b^2+7c^2=4^s(i_1^2+7j_1^2)$ for some integer $s\geq2$ with $i_1\equiv j_1\equiv1\Mod2$ or $i_1\not\equiv j_1\Mod2$.
If $i_1\equiv j_1\equiv1\Mod2$, then 
\begin{equation}\label{4}
4(i_1^2+7j_1^2)=\left(\frac{3i_1\pm7j_1}{2}\right)^2+7\left(\frac{i_1\mp3j_1}{2}\right)^2
\end{equation}
and one of the integers $\frac{i_1-3j_1}{2}$ or $\frac{i_1+3j_1}{2}$ is relatively prime to $6$. 
Therefore, there are integers $b_2,c_2$ such that $b^2+7c^2=b_2^2+7c_2^2$ such that $\text{gcd}(b_2c_2,6)=1$.
Assume $i_1\not\equiv j_1\Mod2$. Note that 
$$
16(i_1^2+7j_1^2)=(3i_1\pm7j_1)^2+7(3j_1\mp i_1)^2.
$$
Therefore, there are integers $b_3,c_3$ such that $b^2+7c^2=b_3^2+7c_3^2$ and $\text{gcd}(b_3c_3,6)=1$ by Equation \eqref{4}. 
Hence, the equation 
$$
x^2+y^2+7z^2+21t^2=24N+30
$$ 
has an integer solution $(x,y,z,t)\in\mathbb{Z}^4$ such that $\text{gcd}(xyzt,6)=1$.
%%%%%%%%%%%%%%%%%%%%%%%%%%%%%%%%%%%%%%%%%%%%%%%%%%
%%%%%%%%%%%%%%%%%%%%%%%%%%%%%%%%%%%%%%%%%%%%%%%%%%%%%%%%%
\vskip 0.5pc
\noindent\textbf{Case (4-2)} $(\alpha_1,\alpha_2,\alpha_3)=(1,1,11)$.
It is enough to show that for any $\alpha$ such that $11\leq\alpha\leq43$ and $\alpha\not\eq22,33$, the equation
$$
x^2+y^2+11z^2+\alpha t^2=24N+13+\alpha
$$ 
has an integer solution $(x,y,z,t)\in\mathbb{Z}^4$ such that $\text{gcd}(xyzt,6)=1$ for any non-negative integer $N$.

Since the proofs are quite similar to each other, we only provide the proof of the case $\alpha=11$. Let $24N+24=11^{2s}\cdot 24n$ for some non-negative integers $n,s$ such that $24n\not\equiv0\Mod{11^2}$. 
If $1\leq n\leq 55$, then one may directly show that the equation 
$$
x^2+y^2+11z^2+11t^2=24n
$$ 
has an integer solution $(x,y,z,t)\in\mathbb{Z}^4$ such that $\text{gcd}(xyzt,6)=1$. 
Therefore, we assume $n\geq56$. 
Note that the genus of $f(x,y,z)=(6x+y)^2+y^2+11z^2$ consists of 
$$
M_f=\langle2,11,18\rangle,~ 
M_2=\begin{pmatrix}5&1&1\\1&7&3\\1&3&13\end{pmatrix},~ M_3=\begin{pmatrix}4&2&2\\2&5&0\\2&0&26\end{pmatrix},~\text{and}~ M_4=\langle2,2,99\rangle.
$$
One may easily show that every non-negative integer congruent to $13$ mod $24$ which is not of the form $k\cdot11^{2l+1}$ for any non-negative integers $l,k$ such that $\left(\frac{k}{11}\right)=-1$ is represented by $M_f$, $M_2$, $M_3$ or $M_4$ by 102:5 of \cite{om}, for it is represented by $M_f$ over $\mathbb{Z}_p$ for any prime $p$. 
Note that 
\begin{equation}\label{34}
M_3\prec_{24,13}M_2\quad \text{and} \quad M_4\prec_{24,13}M_2.
\end{equation}
Therefore, every positive integer congruent to $13$ modulo $24$ which is represented by $M_3$ or $M_4$ is also represented by $M_2$. 

 Let $g$ be the quadratic form associated with $M_2$. We show that every positive integer congruent to $13$ modulo $24$ which is represented by $g$ is also represented by $f$, provided it is not of the form $1309\cdot11^{2l}$ for any non-negative integer $l$.
At first, assume that $g(v)=144n_1+13$ for some non-negative integer $n_1$ and $v\in\mathbb{Z}^3$. 
One may check that there are exactly $55296$ vectors in $R(g,144,13)$ and $464$ matrices in $R(f,g,144)$. Furthermore, all vectors in $R(g,144,13)$ are good vectors with respect to $f,g,144$ and $13$ except $240$ vectors. 
Note  that $R(g,144,13)\setminus R_f(g,144,13)$ is the union of the following two sets. 
$$
\setlength\arraycolsep{0pt}{
\begin{array}{ll}
P_1\!=\!\{(v_1,v_2,v_3)\in R(g,14&4,13) : v_1\equiv3\Mod6,~v_2\equiv\pm4\Mod{24},\\
&v_3\equiv6\Mod{12}, ~ -8v_1+3v_2+16v_3\equiv0\Mod{36}\}\\
P_2\!=\!\{(v_1,v_2,v_3)\in R(g,14&4,13) : v_1\equiv\pm5,\pm11\Mod{24},v_2\equiv\pm4\Mod{24},\\
&v_3\equiv\pm6\Mod{24},~52v_1+43v_2+84v_3\equiv0\Mod{144}\}
\end{array}}
$$   
Now, define
$$
T_1=\begin{pmatrix}-88&6&-184\\112&12&-80\\-8&102&88\end{pmatrix},\quad
T_2=\begin{pmatrix}-92&-101&84\\-88&38&-120\\20&83&132\end{pmatrix}.
$$
Note that $T_i^tM_gT_i=144^2M_g$ for $i=1,2$. 
For any vector $u\in\mathbb{Z}^3$ such that $u\Mod{144}\in P_1$, 
\begin{equation}\label{P_1}
\frac{1}{144}\cdot u T_1^t\in\mathbb{Z}^3 \text{ and } \frac{1}{144}\cdot uT_1^t\Mod{144}\in P_2\cup R_f(g,144,13).
\end{equation}
Similarly, for any vector $u\in\mathbb{Z}^3$ such that $u\Mod{144}\in P_2$, 
\begin{equation}\label{P_2}
\frac{1}{144}\cdot u T_2^t\in\mathbb{Z}^3 \text{ and } \frac{1}{144}\cdot uT_2^t\Mod{144}\in P_1\cup R_f(g,144,13).
\end{equation}
All computations were done by a computer program based on MAPLE.
If $v\Mod{144}$ is a good vector with respect to $f,g,144$ and $13$, then there is a $T\in R(f,g,144)$ such that $\frac{1}{144}\cdot vT^t\in\mathbb{Z}^3\text{ and } f\left(\frac{1}{144}\cdot vT^t\right)=144n_1+13$.
Assume that $v\Mod{144}$ is contained in $R(g,144,13)\setminus R_f(g,144,13)$.
We may further assume that $v\Mod{144}\in P_1$ by \eqref{P_2}.
Then by \eqref{P_1}, we know that 
$$
\frac{1}{144}\cdot vT_1^t\in\mathbb{Z}^3 \text{ and } \frac{1}{144}\cdot vT_1\Mod{144}\in P_2\cup R_f(g,144,13).
$$
If $\frac{1}{144}\cdot vT_1^t\Mod{144}\in R_f(g,144,13)$, then we are done. 
Assume $\frac{1}{144}\cdot vT_1^t\Mod{144}\in P_2$.
Then by \eqref{P_2}, we know that 
$$
\left(\frac{1}{144}\right)^2\cdot vT_1^tT_2^t\in\mathbb{Z}^3 \text{ and } \left(\frac{1}{144}\right)^2\cdot vT_1^tT_2^t\Mod{144}\in P_1\cup R_f(g,144,13).
$$
Now, inductively, there are three possibilities:
\begin{enumerate} 
\item[(i)] there is a positive integer $m$ such that
\begin{small}
$$
\left(\frac{1}{144}\right)^{2m}\cdot v(T_1^tT_2^t)^{m}\in \mathbb{Z}^3 \text{ and }
\left(\frac{1}{144}\right)^{2m}\cdot v(T_1^tT_2^t)^{m}\Mod{144}\in R_f(g,144,13);
$$
\end{small}
\item[(ii)] there is a positive integer $m$ such that 
\begin{small}
$$
\left(\frac{1}{144}\right)^{2m+1}\cdot v(T_1^tT_2^t)^{m}T_1^t\in \mathbb{Z}^3 \text{ and }
\left(\frac{1}{144}\right)^{2m+1}\cdot v(T_1^tT_2^t)^{m}T_1^t\Mod{144}\in R_f(g,144,13);
$$
\end{small}
\item[(iii)] for any positive integer $m$,
\begin{small}
$$
\left(\frac{1}{144}\right)^{2m}\cdot v(T_1^tT_2^t)^{m}\in \mathbb{Z}^3 \text{ and }
\left(\frac{1}{144}\right)^{2m}\cdot v(T_1^tT_2^t)^{m}\Mod{144}\in P_1.
$$  
\end{small}
\end{enumerate}
%either there is a positive integer $m$ such that such that
%$$
%\left(\frac{1}{144^2}\right)^{m}\cdot v(T_1^tT_2^t)^{m}\in \mathbb{Z}^3 \text{ and }
%\left(\frac{1}{144^2}\right)^{m}\cdot v(T_1^tT_2^t)^{m}\Mod{144}\in R_f(g,144,13),
%$$
%or for any positive integer $m$,
%$$
%\left(\frac{1}{144^2}\right)^{m}\cdot v(T_1^tT_2^t)^{m}\in \mathbb{Z}^3 \text{ and }
%\left(\frac{1}{144^2}\right)^{m}\cdot v(T_1^tT_2^t)^{m}\Mod{144}\in P_1
%$$ 
Since there are only finitely many integer solution of $g(x,y,z)=144n_1+13$ and $\left(\frac{1}{144}\right)^2\cdot T_2T_1$ has an infinite order, the latter is impossible unless $v$ is an eigenvector of $T_2T_1$.
Note that $\pm(9,4,6)$ are the only integral primitive eigenvectors of $T_2T_1$.
Hence, if $144n_1+13$ is not of the form $Q(\pm9t,\pm4t,\pm6t)=1309t^2$ for some positive integer $t$, then it is also represented by $f$.
Assume that $144n_1+13=1309t^2$ for a positive integer $t$.
Further assume that $t$ has a prime divisor relatively prime to $2\cdot3\cdot11$. 
Since $\text{gen}(f)=\text{spn}(f)$ and $M_2,M_3$ and $M_4$ represent $1309$, $f$ represents $1309t^2$ by Lemma 2.4 in \cite{poly}.
Note that $f$ represents $1309\cdot2^2$ and $1309\cdot3^2$.\
Therefore if $144n_1+13\not\eq1309\cdot11^{2l}$ for any non-negative integer $l$, then $144n_1+13$ is also represented by $f$. 
Similary, one may prove that every positive integer congruent to $r$ modulo $144$ that is represented by $g$, is also represented by $f$, provided it is not of the form $1309\cdot11^{2l}$, for any $r\in\{37,61,85,109,133\}$.
Then by \eqref{34}, we know that every positive integer congruent to $13$ modulo $24$ which is not 
of the form $k\cdot11^{2l+1}$ and $1309\cdot11^{2l}$, for any non-negative integers $l,k$ such that $\left(\frac{k}{11}\right)=-1$ is represented by $f$.

One may easily show that there is an integer  $d\in\{1,5,7,11\}$ such that $24n-11d^2$ is not of the form $1309\cdot11^{2l}$ and $k\cdot11^{2l+1}$ for any non-negative integers $l,k$ such that $\left(\frac{k}{11}\right)=-1$. 
Furthermore, since we are assuming that $n\geq56$, $24n-11d^2$ is a positive integer.
Therefore, the equation 
$$
x^2+y^2+11z^2=24n-11d^2
$$ 
has an integer solution $(x,y,z)\in\mathbb{Z}^3$ such that $x\equiv y\Mod6$. 
Then one may easily show that $\text{gcd}(xyz,6)=1$.
This completes the proof.
%%%%%%%%%%%%%%%%%%%%%%%%%%%%%%%%%%%%%%%%%%%%%%%%%%
%%%%%%%%%%%%%%%%%%%%%%%%%%%%%%%%%%%%%%%%%%%%%%%%%%%%%%%%%
\vskip 0.5pc
\noindent\textbf{Case (4-3)} $(\alpha_1,\alpha_2,\alpha_3)=(1,2,5)$.
It is enough to show that for any $\alpha$ such that $5\leq\alpha\leq18$ and $\alpha\not\eq6,8$, the equation
$$
x^2+2y^2+5z^2+\alpha t^2=24N+8+\alpha
$$
has an integer solution $(x,y,z,t)\in\mathbb{Z}^4$ such that $\text{gcd}(xyzt,6)=1$ for any non-negative integer $N$. 

Since the proofs are quite similar to each other, we only provide the proof of the case 
$\alpha=15$.
Let $24N+23=5^{2s}(24n+23)$ for some non-negative integers $n,s$ such that $24n+23\not\equiv0\Mod{5^2}$. 
If $0\leq n \leq29$, then one may directly check that the equation
$$
x^2+2y^2+5z^2+15t^2=24n+23
$$ 
has an integer solution $(x,y,z,t)\in\mathbb{Z}^4$ such that $\text{gcd}(xyzt,6)=1$. 
Therefore, we assume $n\geq30$. 
Since $h(\langle1,2,5\rangle)=1$, one may easily show that every non-negative integer not of the form $k\cdot5^{2l+1}$ for any non-negative integers $l,k$ such that $\left(\frac{k}{5}\right)=-1$ is represented by $\langle1,2,5\rangle$. 
One may easily show that there is an integer $d\in\{1,5,7\}$ such that $24n+23-15d^2\not\eq k\cdot5^{2l+1}$.
Furthermore, since we are assuming that $n\geq30$, $24n+23-15d^2$ is a positive integer. 
Therefore, the equation
$$
x^2+2y^2+5z^2=24n+23-15d^2
$$ 
has an integer solution $(x,y,z)=(a,b,c)\in\mathbb{Z}^3$.
Assume $a\equiv b\equiv0\Mod3$.
If  $a^2+2b^2\not\eq0$, then there are integers $a_1,b_1$ such that $a^2+2b^2=a_1^2+2b_1^2$ and $a_1b_1\not\equiv0\Mod3$ by Theorem 9 of \cite{Jones} (see also \cite{Jagy}, and for more generalization see \cite{oy}).
If $a^2+2b^2=0$, then $c\equiv0\Mod4$. Let $c=4c_1$. 
Then $5c^2=80c_1^2=(5c_1)^2+2(5c_1)^2+5c_1^2$ with $c_1\not\equiv0\Mod3$. 
Assume $a\equiv c\equiv0\Mod3$. 
If $a^2+5c^2\not\eq0$, then there are integers $a_2,c_2$ such that $a^2+5c^2=a_2^2+5c_2^2$ and $a_2c_2\not\equiv0\Mod3$ by Theorem 9 of \cite{Jones}.
If $a^2+5c^2=0$, then $b\equiv0\Mod2$. Let $b=2b_2$. Then $2b^2=8b_2^2=b_2^2+2b_2^2+5b_2^2$ with $b_2\not\equiv0\Mod3$. 
Therefore, the equation
$$
x^2+2y^2+5z^2=24n+23-15d^2
$$ 
has an integer solution $(x,y,z)=(a,b,c)\in\mathbb{Z}^3$ such that $abc\not\equiv0\Mod3$. By taking suitable signs of $a,b$ and $c$, we may assume that $a\equiv b\equiv c\Mod3$.
Then either $a\equiv b\equiv c\equiv1\Mod2$ or $a\equiv c\Mod4$ and $b\equiv0\Mod2$.
Assume that  $a\equiv c\Mod4$ and $b\equiv0\Mod2$.
If we define 
$$
\tau=\frac14\begin{pmatrix}3&2&5\\-1&-2&5\\-1&2&1\end{pmatrix},
$$ 
then $\tau(a,b,c)^t=(a_1,b_1,c_1)^t$ is also an integer solution of the equation
\begin{equation}\label{125}
x^2+2y^2+5z^2=24n+23-15d^2.
\end{equation} 
 Note that $a_1b_1c_1\not\equiv0\Mod3$ and $a_1\not\equiv b_1\equiv c_1\Mod3$.
Assume  $a_1\equiv c_1\Mod4$ and $b_1\equiv0\Mod2$. 
Then $\tau(a_1,b_1,c_1)^t=(a_2,b_2,c_2)^t$ is also an integer solution of Equation \eqref{125} such that $a_2b_2c_2\not\equiv0\Mod3$ and $a_2\equiv b_2\equiv c_2\Mod3$.
Therefore, either there is a positive integer $m$ such that $\tau^m(a,b,c)^t=(a_m,b_m,c_m)^t$ is an integer solution of Equation \eqref{125} satisfying $a_mb_mc_m\not\equiv0\Mod3$ and $a_m\equiv b_m\equiv c_m\equiv1\Mod2$ or 
for any positive integer $m$, $\tau^m(a,b,c)^t=(a_m,b_m,c_m)^t$  is an integer solution of Equation \eqref{125}  such that $a_mb_mc_m\not\equiv0\Mod3$, $a_m\equiv c_m\Mod4$ and $b_m\equiv0\Mod2$.
Since there are only finitely many integer solution of Equation \eqref{125} 
and $\tau$ has an infinite order, the latter is impossible unless $(a,b,c)$ is an eigenvector of $\tau$. 
Note that $\pm(0,-5,2)$ are the only integral primitive eigenvectors of $\tau$. 
Since $a\not\equiv0\Mod3$, we may assume that $(a,b,c)$ is not an eigenvector of $\tau$. 
Therefore, the equation 
$$
x^2+2y^2+5z^2=24n+23-15d^2
$$
has an integer solution $(x,y,z)=(a,b,c)\in\mathbb{Z}^3$ such that $\text{gcd}(abc,6)=1$. This completes the proof. 
%%%%%%%%%%%%%%%%%%%%%%%%%%%%%%%%%%%%%%%%%%%%%%%%%%
%%%%%%%%%%%%%%%%%%%%%%%%%%%%%%%%%%%%%%%%%%%%%%%%%%%%%%%%%
\vskip 0.5pc
\noindent\textbf{Case (4-4)} $(\alpha_1,\alpha_2,\alpha_3)=(1,2,7)$.
It is enough to show that for any $\alpha$ such that $7\leq\alpha\leq27$ and $\alpha\not\eq8$, the equation
$$
x^2+2y^2+7z^2+\alpha t^2=24N+10+\alpha
$$ 
has an integer solution $(x,y,z,t)\in\mathbb{Z}^4$ such that $\text{gcd}(xyzt,6)=1$ for any non-negative integer $N$. 

Since the proofs are quite similar to each other, we only provide the proof of the case $\alpha=21$.
Let $24N+31=7^{2s}(24n+7)$ for some non-negative integers $n,s$ such that $24n+7\not\equiv0\Mod{7^2}$. 
For the case when $n=0$, note that for any $s\geq1$
$$
(7^s)^2+2(7^s)^2+7(5\cdot7^{s-1})^2+21(7^{s-1})^2=7^{2s+1}.
$$ 
If $1\leq n\leq 105$, one may directly check that the equation
$$
x^2+2y^2+7z^2+21t^2=24n+7
$$ 
has an integer solution $(x,y,z,t)\in\mathbb{Z}^4$ such that $\text{gcd}(xyzt,6)=1$. 
Therefore, we assume $n\geq106$. 
Note that the genus of $f(x,y,z)=(3x+y)^2+2y^2+7z^2$ consists of 
$$
M_f=\langle3,6,7\rangle\quad\text{and}\quad M_2=\langle1\rangle\perp\begin{pmatrix}9&3\\3&15\end{pmatrix}.
$$
One may easily show that every non-negative integer congruent to $1$ modulo $3$ which is not of the form $k\cdot7^{2l+1}$ for any non-negative integers $l,k$ such that $\left(\frac{k}{7}\right)=-1$ is represented by $M_f$ or $M_2$ by 102:5 of \cite{om}, for it is represented by $M_f$ over $\mathbb{Z}_p$ for any prime $p$.
Note that $M_2\prec_{8,2}M_f$. 
One may easily show that there is an integer $d\in\{1,5,7,11\}$ such that $24n+7-21d^2\equiv1\Mod3$, $24n+7-21d^2\equiv2\Mod8$ and $24n+7-21d^2\not\eq k\cdot7^{2l+1}$.
Furthermore, since we are assuming that $n\geq106$, $24n+7-21d^2$ is a positive integer.
Therefore, the equation
$$
x^2+2y^2+7z^2=24n+7-21d^2
$$ 
has an integer solution $(x,y,z)=(a,b,c)\in\mathbb{Z}^3$ such that $a\equiv b\Mod3$ by \eqref{gt}.
Assume $a\equiv b\equiv0\Mod3$.
Since $7c^2\not\equiv2\Mod{8}$, we know that $a^2+2b^2\not\eq0$. 
Then there are integers $a_1,b_1$ such that $a^2+2b^2=a_1^2+2b_1^2$ and $a_1b_1\not\equiv0\Mod3$ by Theorem 9 of \cite{Jones}.
Therefore, the equation 
$$
x^2+2y^2+7z^2=24n+7-21d^2
$$ 
has an integer solution $(x,y,z)=(a,b,c)\in\mathbb{Z}^3$ such that $abc\not\equiv0\Mod3$.
Note that $b\equiv1\Mod2$ and either $a\equiv c\equiv1\Mod2$ or $a\equiv c\equiv0\Mod2$ and $a\equiv c\Mod4$. 
Similarly as in the proof of  Case (4-1), there are integers $a_2,c_1$ such that $a^2+7c^2=a_2^2+7c_1^2$ with $\text{gcd}(a_2c_1,6)=1$.
This completes the proof.
%%%%%%%%%%%%%%%%%%%%%%%%%%%%%%%%%%%%%%%%%%%%%%%%%%
%%%%%%%%%%%%%%%%%%%%%%%%%%%%%%%%%%%%%%%%%%%%%%%%%%%%%%%%%
\vskip 0.5pc
\noindent\textbf{Case (4-5)} $(\alpha_1,\alpha_2,\alpha_3)=(1,3,5)$.
It is enough to show that for any $\alpha$ such that $5\leq\alpha\leq19$ and $\alpha\not\eq6,7,8,9$, the equation
$$
x^2+3y^2+5z^2+\alpha t^2=24N+9+\alpha
$$ 
has an integer solution $(x,y,z,t)\in\mathbb{Z}^4$ such that $\text{gcd}(xyzt,6)=1$ for any non-negative integer $N$.

Since the proofs are quite similar to each other, we only provide the proof of the case $\alpha=15$. 
Let $24N+24=5^{2s}\cdot24n$ for some non-negative integers $n,s$ such that $24n\not\equiv0\Mod{5^2}$.
If $1\leq n\leq 30$, then one may directly check that the equation 
$$
x^2+3y^2+5z^2=24n
$$ 
has an integer solution $(x,y,z,t)\in\mathbb{Z}^4$ such that $\text{gcd}(xyzt,6)=1$.
Therefore, we assume $n\geq31$.
Note that the genus of $f(x,y,z)=(2x+z)^2+3(2y+z)^2+5z^2$ consists of 
$$
M_f=\begin{pmatrix}4&2&2\\2&9&3\\2&3&9\end{pmatrix}\quad\text{and}\quad M_2=\langle1\rangle\perp\begin{pmatrix}8&4\\4&32\end{pmatrix}.
$$
One may easily show that every non-negative integer congruent to $1$ modulo $8$ which is not of the form $k\cdot5^{2l+1}$ for any non-negative integers $l,k$ such that $\left(\frac{k}{5}\right)=-1$ is represented by $M_f$ or $M_2$ by 102:5 of \cite{om}, for it is represented by $M_f$ over $\mathbb{Z}_p$ for any prime $p$.
Note that $M_2\prec_{24,9} M_f$.
One may easily show that there is an integer $d\in\{1,5,7\}$ such that $24n-15d^2\equiv9\Mod{24}$ and $24n-15d^2\not\eq k\cdot5^{2l+1}$ for any non-negative integers $l,k$ such that $\left(\frac{k}{5}\right)=-1$.
Furthermore, since we are assuming that $n\geq31$, $24n-15d^2$ is a positive integer.
Therefore, the equation 
$$
x^2+3y^2+5z^2=24n-15d^2
$$ 
has an integer solution $(x,y,z)=(a,b,c)\in\mathbb{Z}^3$ such that $a\equiv b\equiv c\equiv1\Mod2$ by \eqref{gt}.
Assume $a\equiv c\equiv0\Mod3$. Since $a\equiv c\equiv1\Mod2$, $a^2+5c^2\not\eq0$.
By Theorem 9 of \cite{Jones}, there are integers $a_1,c_1$ such that $a^2+5c^2=a_1^2+5c_1^2$ with $a_1c_1\not\equiv0\Mod3$.
Furthermore, $a_1\equiv c_1\equiv1\Mod2$.
Therefore, the equation 
$$
x^2+3y^2+5z^2=24n-15d^2
$$ 
has an integer solution $(x,y,z)=(a,b,c)\in\mathbb{Z}^3$ such that $a\equiv b\equiv c\equiv1\Mod2$ and $ac\not\equiv0\Mod3$.
If $b\not\equiv0\Mod3$, then we are done.
Assume $b\equiv0\Mod3$. 
Since $a\equiv b\equiv1\Mod2$, $a^2+3b^2\not\eq0$.
Then 
$$
a^2+3b^2=\left(\frac{a\pm3b}{2}\right)^2+3\left(\frac{a\mp b}{2}\right)^2,
$$
where $\left(\frac{a\pm3b}{2}\right)\cdot\left(\frac{a\mp b}{2}\right)\not\equiv0\Mod3$ and one of the integers $\frac{a-b}{2}$ or $\frac{a+b}{2}$ is odd. 
Therefore, the equation
$$
x^2+3y^2+5z^2+15t^2=24N+24
$$ 
has an integer solution $(x,y,z,t)\in\mathbb{Z}^4$ such that $\text{gcd}(xyzt,6)=1$. This completes the proof.\end{proof}

\begin{thm}\label{quinary}
There are exactly $124$ quinary proper universal sums of generalized pentagonal numbers.
\end{thm}

\begin{proof}
For positive integers $\alpha_1,\alpha_2,\dots,\alpha_5$, let $\Phi_{\alpha_1,\alpha_2,\dots,\alpha_5}$ be a proper universal sum of generalized pentagonal numbers. From the above theorem, 
we may assume that $(\alpha_1,\alpha_2,\alpha_3,\alpha_4)=(1,1,11,22)$ or $(1,1,11,33)$.
Now, one may directly check that the remaining candidates $\Phi_{1,1,11,22}$ and $\Phi_{1,1,11,33}$ are not universal and their truants are $98$ and $109$, respectively.
If $(\alpha_1,\alpha_2,\alpha_3,\alpha_4)=(1,1,11,22)$, then $\alpha_5\leq98$, for the truant of $\Phi_{1,1,11,22}$ is $98$ and $98\rightarrow\Phi_{1,1,11,22,\alpha_5}$.  
Similarly, if $(\alpha_1,\alpha_2,\alpha_3,\alpha_4)=(1,1,11,33)$, then $\alpha_5\leq109$.

%%%%%%%%%%%%%%%%%%%%%%%%%%%%%%%%%%%%%%%%%%%%%%%%%%
%%%%%%%%%%%%%%%%%%%%%%%%%%%%%%%%%%%%%%%%%%%%%%%%%%%%%%%%%
\vskip 0.5pc
\noindent\textbf{Case (5-1)} $(\alpha_1,\alpha_2,\alpha_3,\alpha_4)=(1,1,11,22)$.
It is enough to show that for any $\alpha$ such that $\alpha=22,33$ or $44\leq\alpha\leq98$, the equation 
$$
x^2+y^2+11z^2+22t^2+\alpha s^2=24N+35+\alpha
$$ 
has an integer solution $(x,y,z,t,s)\in\mathbb{Z}^5$ such that $\text{gcd}(xyzts,6)=1$ for any non-negative integer $N$.
Similarly as in the proof of Case (4-2), one may show that the equation
$$
x^2+y^2+11z^2+22t^2=24N+35
$$
has an integer solution $(x,y,z,t)\in\mathbb{Z}^4$ such that $\text{gcd}(xyzt,6)=1$ for any non-negative integer $N$ except $98$.
The proof follows immediately.
%%%%%%%%%%%%%%%%%%%%%%%%%%%%%%%%%%%%%%%%%%%%%%%%%%
%%%%%%%%%%%%%%%%%%%%%%%%%%%%%%%%%%%%%%%%%%%%%%%%%%%%%%%%%
\vskip 0.5pc
\noindent\textbf{Case (5-2)} $(\alpha_1,\alpha_2,\alpha_3,\alpha_4)=(1,1,11,33)$.
It is enough to show that for any $\alpha$ such that $\alpha=33$ or $44\leq\alpha\leq109$, the equation 
$$
x^2+y^2+11z^2+33t^2+\alpha s^2=24N+46+\alpha
$$ 
has an integer solution $(x,y,z,t,s)\in\mathbb{Z}^5$ such that $\text{gcd}(xyzts,6)=1$ for any non-negative integer $N$.
Similarly as in the proof of Case (4-2), one may show that if 
$24N+46\not\eq2\cdot11^{2s+1}$ for any $s\geq1$, then the equation 
$$
x^2+y^2+11z^2+33t^2=24N+46
$$ 
has an integer solution $(x,y,z,t)\in\mathbb{Z}^4$ such that $\text{gcd}(xyzt,6)=1$.
If $24N+46\not\eq 2\cdot11^{2s+1}$ for any integer $s\geq1$, then the equation
$$
x^2+y^2+11z^2+33t^2=24N+46+\alpha-\alpha\cdot1^2
$$ 
has an integer solution  $(x,y,z,t)\in\mathbb{Z}^4$ such that $\text{gcd}(xyzt,6)=1$.
If $24N+46=2\cdot11^{2s+1}$ for some integer $s\geq1$, then the equation
$$
x^2+y^2+11z^2+33t^2=24N+46+\alpha-\alpha\cdot5^2
$$ 
has an integer solution  $(x,y,z,t)\in\mathbb{Z}^4$ such that $\text{gcd}(xyzt,6)=1$ since $24N+46+\alpha-\alpha\cdot5^2\not\equiv0\Mod{11^2}$. 
This completes the proof.
\end{proof}

We provide an effective criterion on the universality of an arbitrary sum of generalized pentagonal numbers.
\begin{thm} For a positive integer $k$,
let $\alpha_1,\alpha_2,\dots,\alpha_k$ be positive integers. An arbitrary sum 
$$
\Phi_{\alpha_1,\alpha_2,\dots,\alpha_k}(x_1,x_2,\dots,x_k)=\alpha_1P_5(x_1)+\alpha_2P_5(x_2)+\cdots+\alpha_kP_5(x_k)
$$ 
of generalized pentagonal numbers is universal if and only if it represents
$$
1,\ 3,\ 8,\ 9,\ 11,\ 18,\ 19,\ 25,\ 27,\ 43,\ 98,\ \text{and}\ 109.
$$ 
\end{thm}

\begin{proof} 
Without loss of generality, we may assume that $\alpha_1\leq\alpha_2\leq \cdots\leq\alpha_k$.
If the sum $\Phi_{\alpha_1,\alpha_2,\dots,\alpha_k}$ of generalized pentagonal numbers represents above $12$ integers, then one may easily show that there is an integer $u$ such that $1\leq u\leq k$ and the partial sum $\Phi_{\alpha_1,\alpha_2,\dots,\alpha_u}$ of $\Phi_{\alpha_1,\alpha_2,\dots,\alpha_k}$ is one of the sums in Table 3.1. 
Then the partial sum $\Phi_{\alpha_1,\alpha_2,\dots,\alpha_u}$ of $\Phi_{\alpha_1,\alpha_2,\dots,\alpha_k}$ is universal by \eqref{pent} and Theorems \ref{quaternary} and \ref{quinary}. 
Therefore, the sum $\Phi_{\alpha_1,\alpha_2,\dots,\alpha_k}$ of generalized pentagonal numbers is universal.
\end{proof}

\end{document}